\documentclass[a4paper,12pt,reqno]{amsart}
\usepackage{amsmath,amsthm,amssymb}
\usepackage[numbers,sort&compress]{natbib}
\usepackage[usenames, dvipsnames]{color}
\nonstopmode \numberwithin{equation}{section}

\newtheorem{theorem}{Theorem}[section]

\newtheorem{lemma}{Lemma}[section]

\newtheorem{remark}{Remark}[section]

\allowdisplaybreaks
\DeclareMathOperator{\real}{Re}

\allowdisplaybreaks

\begin{document}

\title{Inequalities of extended Beta and extended hypergeometric functions}

%\titlerunning{Bessel-Struve kernel function}        % if too long for running head

\author{Saiful R. Mondal}


\address{Saiful R. Mondal, 
              Department of Mathematics and Statistics, Collage of Science,
King Faisal University, Al-Hasa 31982, Hofuf, Saudi Arabia.
              } \email{smondal@kfu.edu.sa}   


\begin{abstract}
This paper studies the log-convexity of the extended beta functions.  As a consequence, Tur\'an-type inequalities  are established.The monotonicity, log-convexity, log-concavity of extended hypergeometric functions are deduced  by using the inequalities on extended beta functions,. The particular
cases of those results also gives the Tur\'an-type inequalities for extended confluent and extended Gaussian hypergeometric functions. Some reverse of Tur\'an-type inequalities also derived.


\end{abstract}

\maketitle

% Short title is optional, it will appear in running heads.
% It is necessary only if the title is to long to be used in running heads

        % Please, write the date of submission

        % AMS-2010 subj class. The list can be found on http://www.ams.org/mathscinet/msc/msc2010.html

%\thanks{This research is supported by ...\par This paper is a lecture that was given at...}
        % Optional. Only one command thanks is allowed, use \par inside text if you need multiple thanks.

\section{Introduction }

For $\real(x)>0$,  $\real(y)>0$, $\real(\sigma)>0$, define the functions
\begin{align}\label{eqn:ext-beta-1}
\mathtt{B}_\sigma(x,y):= \int_0^1 t^{x-1} (1-t)^{y-1} exp \left(- \tfrac{\sigma}{t(1-t)}\right) dt.
\end{align}
The function $\mathtt{B}_\sigma$ is known as the extended beta functions which was introduced by   Chaudhry\ et al. \cite{Chaudhry-1}.  They  discussed several properties of this extended beta functions  and also established  connection  with the Macdonald, error and Whittaker functions (also see \cite{Miller}).

Later, using this extended beta functions,
an extended confluent hypergeometric functions(ECHF)  can be seen  in  the article by   Chaudhry\ et al.  \cite{Chaudhry-2} .  The series representation of  the extended confluent hypergeometric functions is
\begin{align}\label{eqn:ECHF-1}
\Phi_\sigma(b ; c; x) := \sum_{n=0}^\infty \frac{ B_\sigma(b+n, c+n)}{B(b, c-b)}  \frac{ x^n}{n!} ,
\end{align}
where $\sigma \geq 0$, $\real(c) > \real(b) >0$.  For $ \sigma >0$, the series converges for all $x$  provided $c \neq 0, -1, -2, \ldots$.

The ECHF can also have the integral representration as
\begin{align}\label{eqn:ECHF-integral}
\Phi_\sigma(b ; c; x) :=\frac{1}{B(b, c-b)}  \int_0^1 t^{b-1} (1-t)^{c-b-1} exp \left(x t- \tfrac{\sigma}{t(1-t)}\right) dt.
\end{align}

Similarly,  the extended Gaussian hypergometric functions (EGHF) can be difined by
\begin{align}\label{eqn:EGHF-1}
F_\sigma(a, b ; c; x) := \sum_{n=0}^\infty \frac{ B_\sigma(b+n, c-b)}{B(b, c-b)}  \frac{(a)_n}{n!}  x^n
\end{align}
where $\sigma \geq 0$, $\real(c) > \real(b) >0$ and  $|x|<1$.  For $ \sigma >0$, the series converges when $ | x|< 1$  and  $c \neq 0, -1, -2, \ldots$.

The EGHF  also have the integral form
\begin{align}\label{eqn:EGHF-integral}
F_\sigma(a, b ; c; x) :=\frac{1}{B(b, c-b)}  \int_0^1 t^{b-1} (1-t)^{c-b-1} (1-x t) ^{-a} exp \left(- \tfrac{\sigma}{t(1-t)}\right) dt.
\end{align}
Note that  for $p=0$, the series \eqref{eqn:ECHF-1} and  \eqref{eqn:EGHF-1}  respectively    reduces  to the  classical confluent hypergeometric series and the  Gaussian  hypergeometric series.

The aim of this article is to study the log-convexity and log-convexity of the above three extended functions. In particular, we will give more emphasis on the Tur\'an type inequality \cite{Turan}  and it's reverse form.

The work here is motivated from the resent works, like as   \cite{Ba1, Ba, BP, BP2, KS, BGR},  of  many researcher in this direction  and references their in. Inequalities related to beta functions with the significance in this study can be found in  \cite{Dragomir-1, Agarwal}.

In Section $2.1$,  we state and prove several inequalities for extended Beta functions. The  classical Chebyshev's integral inequality and the  H$\ddot{\text{o}}$lder-Rogers inequality for integrals are used to obtain the main results in this section.  The results in  the Section $2.1$ are very much useful to generate
inequalities for ECHF and EGHF, especially the Tur\'an type inequality in Section $2.2$.  The log-convexity and log-convexity of   ECHF and EGHF are also given in Section $2.2$.

\section{Results and Discussion}
\subsection{ Inequalities for extended beta functions}
In this section, applying classical integral inequalities like as  Chebychev’s inequality for synchronous
and asynchronous mappings, H$\ddot{\text{o}}$lder-Rogers inequality, we will derive several inequalities for
extended beta functions. Few inequalities are useful in sequel to  derive the Tur\'an type inequalities for EGHF and  ECHF.

\begin{theorem}\label{thm-1} Let $x, y, x_1, y_1 > 0 $ such that $(x-x_1)(y-y_1) \geq 0 $, then
\begin{align}\label{eqn:exb-1}
  B_\sigma(x,  y_1)B_\sigma(x_1,  y) \leq B_\sigma( x_1, y_1) B_\sigma (x, y),
\end{align}
for all $ \sigma \geq 0$.
\end{theorem}

\begin{proof} To prove the result we need to recall the classical Chebyshev's integral inequality \cite[p. 40]{Mi}  : If $f, g : [a, b] \to \mathbb{R}$  are synchronous (both increase or both decrease) integrable functions, and $p: [a, b] \to \mathbb{R}$ is a positive integrable function, then
\begin{align}\label{eqn:chebyshev-1}
\int_a^b p(t) f(t) dt \int_a^b p(t) g(t) dt \leq \int_a^b p(t) dt \int_a^b p(t) f(t) g(t) dt.
\end{align}
The inequality \eqref{eqn:chebyshev-1} is reversed if $f$ and $g$ are asynchronous.

Consider the functions $f(t):=t^{x-x_1}$ ,  $g(t):=t^{y-y_1}$ and  \[p(t):=  t^{x_1-1} (1-t)^{y_1-1} exp\left( -\frac{ \sigma}{(t(1-t)}\right).\]
Clearly,  $p$ is non-negative on $[0,1]$.  Since   $(x-x_1)(y-y_1) \geq 0 $,  it follows that $f'(t)=(x-x_1) t^{x-x_1-1}$ ,  $g'(t)=(y-y_1)t^{y-y_1-1}$ have the same monotonicity on $[0,1]$.

Applying Chebyshev's integral inequality  \eqref{eqn:chebyshev-1}, for the  above selected $f$, $g$ and $p$,  we have
\begin{align*}
&\left(\int_0^1 t^{x-1} (1-t)^{y_1-1} exp\left( -\frac{ \sigma}{(t(1-t)}\right) dt\right)\left( \int_a^b t^{x_1-1} (1-t)^{y-1} exp\left( -\frac{ \sigma}{(t(1-t)}\right)dt\right)\\
&\quad  \quad \leq \left(\int_a^bt^{x_1-1} (1-t)^{y_1-1} exp\left( -\frac{ \sigma}{(t(1-t)}\right) dt\right)
 \left(\int_a^b t^{x-1} (1-t)^{y-1} exp\left( -\frac{ \sigma}{(t(1-t)}\right) dt\right),
\end{align*}
which is equivalent to \eqref{eqn:exb-1}.
\end{proof}

\begin{theorem}\label{thm-2}
The function $\sigma \mapsto B_\sigma(x, y)$  is  log-convex on $(0, \infty)$ for each fixed $x, y>0$.
In particular,
\begin{itemize}
\item[(i)] The functions $B_\sigma(x, y)$ satisfy the Tur\'an type inequality
\begin{align*}
B_{\sigma}^2(x, y) - B_{\sigma+a}(x, y)
B_{\sigma-a}(x, y)\leq 0,
\end{align*}
for all real $a$.  This will further reduce to $B_{\sigma}^2(x, y) \leq B(x, y) B_{2\sigma}(x, y)$  when $\sigma=a$. Here $B(x, y)=B_0(x, y)$ is the well-known classical beta function.
\item[(ii)]  The function $\sigma \mapsto    B_{\sigma}(x-1, y-1)/ B_{\sigma}(x, y)$ is decreasing on $(0, \infty)$ for fixed $x ,y >0$.
\end{itemize}
\end{theorem}
\begin{proof}
From the definition of log-convexity, it is required  to prove that
\begin{align}\label{eqn:ext-beta-lc}
B_{\alpha \sigma_1+(1-\alpha) \sigma_2}(x, y) \leq \left(B_{\sigma_1}(x, y) \right)^\alpha \left(B_{\sigma_2}(x, y) \right)^{1-\alpha},
\end{align}
for $\alpha \in [0,1]$, $\sigma_1, \sigma_2 >0$ and fixed $x, y>0$.

Clearly,  \eqref{eqn:ext-beta-lc} is trivially true for $\alpha = 0$ and $\alpha=1$.

Let $\alpha \in (0,1)$.
 It follows from \eqref{eqn:ext-beta-1}
that
\begin{align}\label{eqn:ext-beta2}\notag
B_{\alpha \sigma_1+(1-\alpha) \sigma_2}(x, y) &= \int_0^1 t^{x-1} (1-t)^{y-1} exp\left(-\frac{\alpha \sigma_1+(1-\alpha)\sigma_2}{t(1-t)}\right) dt\\ \notag
&=\int_0^1 \left(t^{x-1} (1-t)^{y-1} exp\left(-\frac{ \sigma_1}{t(1-t)}\right) dt\right)^\alpha\\
&\quad \quad \times \int_0^1 \left(t^{x-1} (1-t)^{y-1} exp\left(\frac{-\sigma_2}{t(1-t)}\right) dt\right)^{1-\alpha}.
\end{align}
Let $p=1/\alpha$ and $q=1/(1-\alpha)$. Clearly $p>1$ and $p+q=pq$.  Thus,  applying
the well-known H$\ddot{\text{o}}$lder-Rogers inequality for integrals, \eqref{eqn:ext-beta2} yields
\begin{align}\label{eqn:ext-beta3}\notag
B_{\alpha \sigma_1+(1-\alpha) \sigma_2}(x, y)
&< \left(\int_0^1 t^{x-1} (1-t)^{y-1} exp\left(-\frac{ \sigma_1}{t(1-t)}\right) dt\right)^\alpha\\ \notag
&\quad \quad \times \left(\int_0^1 t^{x-1} (1-t)^{y-1} exp\left(-\frac{\sigma_2}{t(1-t)}\right) dt\right)^{1-\alpha}\\
&=\left(B_{\sigma_1}(x, y) \right)^\alpha \left(B_{\sigma_2}(x, y) \right)^{1-\alpha}.
\end{align}
This implies $\sigma \mapsto B_{\sigma}(x, y)$ is log-convex.

Choosing $\alpha=1/2$, $\sigma_1=\sigma-a$ and $\sigma_2=\sigma+a,$ the inequality \eqref{eqn:ext-beta3} gives
\begin{align*}
B_{\sigma}^2(x, y) - B_{\sigma+a}(x, y)
B_{\sigma-a}(x, y)\leq 0.
\end{align*}

The log-convexity of  $B_\sigma(x,y)$ is equivalent to
\begin{align}\label{eqn:iden-dif} \frac{\partial}{\partial \sigma} \left( \frac{ \frac{\partial}{\partial\sigma} B_\sigma(x, y)}{B_\sigma(x, y)}\right) \geq 0.\end{align}
Now the identity \cite[Page 22]{Chaudhry-1}
\[ \frac{\partial^n}{\partial \sigma^n} B_\sigma(x, y) = (-1)^n B_\sigma(x-n, y-n); \quad  n=0, 1, 2, \cdots,\]
reduces \eqref{eqn:iden-dif} to
\[  \frac{\partial}{\partial \sigma} \left( \frac{ B_\sigma(x-1, y-1)}{B_\sigma(x, y)}\right) \leq 0.\]
Hence the conclusion.
\end{proof}

\begin{theorem}\label{thm-3}
The function $(x,y) \mapsto  B_\sigma( x, y)$ is logarithmic convex on $(0, \infty) \times (0, \infty) $,
for all $ \sigma \geq 0$. In particular,
\[ B_\sigma^2\left( \frac{x_1+x_2}{2}, \frac{y_1+y_2}{2}\right) \leq B_\sigma\left(x_1,y_1\right)  B_\sigma\left( x_2, y_2\right).\]
\end{theorem}
\begin{proof}
Let $\alpha_1, \alpha_2 >0$ such that $\alpha_1+\alpha_2=1$. Then for $\sigma \geq 0$, we have
\begin{align*}
&B_\sigma(\alpha_1( x_1, y_1)+\alpha_2(x_2,y_2) ) \\
&\quad = \int_0^1 t^{ \alpha_1x_1+ \alpha_2 x_2-1}  (1-t)^{  \alpha_1 y_1+ \alpha_2 y_2-1} exp\left( -\frac{\sigma}{t(1-t)}\right) dt\\
&\quad =\int_0^1\left(t^{ x_1-1}  (1-t)^{  y_1-1} exp\left( -\frac{\sigma}{t(1-t)}\right)\right)^{\alpha_1} \left(t^{ x_2-1}  (1-t)^{  y_2-1} exp\left( -\frac{\sigma}{t(1-t)}\right)\right)^{\alpha_2} dt\\
\end{align*}
Again by considering $p=1/\alpha_1$, $q=1/\alpha_2$, we can use the  H$\ddot{\text{o}}$lder-Rogers inequality for integrals and it follows that
\begin{align*}
B_\sigma(\alpha_1( x_1, y_1)+\alpha_2(x_2,y_2) )
&\leq \left(\int_0^1 t^{ x_1-1}  (1-t)^{  y_1-1} exp\left( -\frac{\sigma}{t(1-t)}\right) dt \right)^{\alpha_1} \\
&\quad  \times  \left(\int_0^1 t^{ x_2-1}  (1-t)^{  y_2-1} exp\left( -\frac{\sigma}{t(1-t)}\right) dt\right)^{\alpha_2}\\
&= B_\sigma( x_1, y_1)^{\alpha_1}  B_\sigma( x_2, y_2)^{\alpha_2}.
\end{align*}

For $\alpha_1=\alpha_2=1/2$, the above inequality reduces to
\begin{align}\label{eqn1}
 B_\sigma^2\left( \frac{x_1+x_2}{2}, \frac{y_1+y_2}{2}\right) \leq B_\sigma\left(x_1,y_1\right)  B_\sigma\left( x_2, y_2\right).
\end{align}
 Let $x, y >0 $ such that $\displaystyle \min_{a \in \mathbb{R}}(x+a, x-a) >0$, then $x_1 = x+a$ , $x_2=x-a$ and $y_1=y+b$,  $y_2=y-b$ in  \eqref{eqn1}  yields
\begin{align}
[ B_\sigma(x, y)]^2 \leq B_\sigma( x+a, y+b) B_\sigma (x-a, y-b) ,
\end{align}
for all $ \sigma \geq 0$.
\end{proof}
 The  Gr$\ddot{\text{u}}$ss' inequality \cite[ Page 295-310]{Mi2} for the integrals given in the following lemma.
\begin{lemma} \label{lemma-Gruss}
Let $f$ and $g$ be two integrable functions on $[a, b]$.  If
\[ m \leq f(t) \leq M  \quad \text{and} \quad  l \leq g(t) \leq L , \quad  \text {for each} \quad   t \in [a, b]; \]
where $m$, $M$, $l$, $L$ are given real constant .  Then
\begin{align}
\left|D(f, g; h) \right|\leq D(f, f; h)^{1/2}D(g, g; h)^{1/2} \leq  \frac{1}{4} (M-m) (L-l) \left[ \int_a^b h(t) dt \right]^2,
\end{align}
where
\begin{align*}
D(f, g; h) :=  \int_a^b h(t) dt \int_a^b h(t) f(t) g(t)  dt-   \int_a^b h(t) f(t) dt \int_a^b h(t)  g(t)  dt.
\end{align*}
\end{lemma}
Our next result is the application of the Gr$\ddot{\text{u}}$ss' inequality  for the extended beta mappings.
\begin{theorem}
Let  $\sigma_1, \sigma_2, x, y >0$ . Then
\begin{align} \label{eqn:Gruss-ext-beta} \notag
&\left| B_{\sigma_1+\sigma_2}(x+y+1, x+y+1) -  B_{\sigma_1}(x+1, x+1)  B_{\sigma_2}(y+1, y+1) \right| \\ \notag
&\quad\leq \left[ B_{2\sigma_1}(2x+1, 2x+1) -  B_{\sigma_1}(x+1, x+1) ^2\right]^{\frac{1}{2}} \left[ B_{2\sigma_1}(2y+1, 2y+1) -  B_{\sigma_1}(y+1, y+1) ^2\right]^{\frac{1}{2}}\\
&\quad \leq \frac{  exp(-4(\sigma_1+\sigma_2)}{4^{x+y+1}}.
\end{align}
\end{theorem}

\begin{proof} To prove the inequality, it is required to determine the  upper and lower bounds of
\begin{align*}
f(t) &: = t^x (1-t)^x exp \left( -\frac{\sigma_1}{t(1-t)}\right)\\
g(t) &: = t^y (1-t)^y exp \left( -\frac{\sigma_2}{t(1-t)}\right),
\end{align*}
for $t \in [0,1]$, and $x, y , \sigma_1, \sigma_2 >0$. Clearly,  $f(0)=f(1)=0$ and   $g(0)=g(1)=0$. Now for $t \in (0,1)$, the logarithmic differentiation of $f$ yields
\begin{align*}
f'(t) = f(t) (1-2t) \left( \frac{ x t(1-t)+\sigma_1}{t^2 (1-t)^2}\right).
\end{align*}
Since   $f(t)>0 $  and $x t(1-t)+\sigma_1>0$ on  $t \in (0,1)$, $f'(t) >0$  for  $t>1/2$  and  $f'(t) <0$  for  $t<1/2$. This implies
\[ M = \frac{exp(- 4 \sigma_1)}{4^x}.\]
Similarly, it can be shown that
\[L= \frac{exp(- 4 \sigma_2)}{4^y}.\]
Now setting $f$, $g$ as above and  $h(t)=1$ for all $t \in [0,1]$ in Lemma \ref{lemma-Gruss} gives  \eqref{eqn:Gruss-ext-beta}.
\end{proof}

\begin{remark}
Consider the functions
\[ f(t) = t^x,  \quad  g(t)= (1-t)^y \quad  \text{and} \quad h(t) =  t^{x_1-1} (1-t)^{y_1-1}  exp\left(- \frac{\sigma}{t(1-t)}\right),\]
for $t \in [0, 1] $,  $x, y, x_1, y_1 >0$. Clearly  $M=L=1$ and $m=l=0$.  Thus we have the following inequality from Lemma  \ref{lemma-Gruss}:

\begin{align} \label{eqn:Gruss-ext-beta-2} \notag
&\left| B_{\sigma}(x_1, y_1)B_{\sigma}(x+x_1, y+y_1) -  B_{\sigma}(x+x_1, y_1)   B_{\sigma}(x_1, y+y_1) \right| \\ \notag
&\quad\leq \left[B_{\sigma}(x_1, y_1)  B_{\sigma}(2x+x_1, y_1)  -  B^2_{\sigma}(x+x_1, y_1) \right]^{\frac{1}{2}} \\ \notag
& \quad \quad \times
 \left[B_{\sigma}(x_1, y_1)  B_{\sigma}(x_1, 2y+y_1)  -  B^2_{\sigma}(x_1,y+ y_1) \right]^{\frac{1}{2}}\\
&\quad \leq \frac{ B^2_{\sigma}(x_1, y_1)}{4}.
\end{align}

Similarly, if  $f$, $g$ and $h$ defined  as
\[ f(t): = t^m(1-x)^n,  \quad  g(t):= t^p(1-t)^q \quad  \text{and} \quad h(t) :=  t^{\alpha-1} (1-t)^{\beta-1}  exp\left(- \frac{\sigma}{t(1-t)}\right),\]
for $t \in [0,1]$ and $\alpha, \beta, m, n, p, q>0 $, then  ( see \cite{Dragomir-1})   we have
\[ M= \frac{ m^m n^n}{(m+n)^{m+n}} \quad \text{and} \quad L=\frac{ p^p q^q}{(p+q)^{p+q}},\]
hence  the inequality
\begin{align} \label{eqn:Gruss-ext-beta-3} \notag
&\left| B_{\sigma}(\alpha, \beta )B_{\sigma}(\alpha+m+p, \beta+n+q) -  B_{\sigma}(\alpha+m, \beta+n)   B_{\sigma}(\alpha+p, \beta+q) \right| \\ \notag
&\quad\leq \left[B_{\sigma}(\alpha, \beta )  B_{\sigma}(\alpha+2m, \beta+2n)  -  B^2_{\sigma}(\alpha+m, \beta+m) \right]^{\frac{1}{2}} \\ \notag
& \quad \quad \times
 \left[B_{\sigma}(\alpha, \beta )  B_{\sigma}(\alpha+2p, \beta+2q)  -  B^2_{\sigma}(\alpha+p, \beta+q) \right]^{\frac{1}{2}} \\
&\quad \leq \frac{ B^2_{\sigma}(\alpha, \beta )}{4}. \frac{ m^m n^n}{(m+n)^{m+n}}.  \frac{ p^p q^q}{(p+q)^{p+q}},
\end{align}
follows from  Lemma \ref{lemma-Gruss} .

It is evident from Theorem \ref{thm-1},    the inequalities  \eqref{eqn:Gruss-ext-beta-2}  and \eqref{eqn:Gruss-ext-beta-3}, that the results  discussed in \cite{Agarwal, Dragomir-1} for classical beta functions can be replicated for the  extended beta functions. Hence, we can restrict our work which have some influence of the parameter $\sigma$,  like as Theorem \ref{thm-2},  Theorem \ref{thm-4} . However,  Theorem \ref{thm-1} and Theorem \ref{thm-3} are useful to get the inequalities for ECHF and EGHF in the next section.
\end{remark}

\subsection{Inequalities for ECHF and EGHF}

Along with the integral inequalities mention in the previous section, the  following result of Biernacki and Krzy\.z \cite{BK} will be used in the sequel.

\begin{lemma}\label{lemma:1}\cite{BK}
Consider the power series $f(x)=\sum_{n\geq0} a_n x^n$ and $g(x)=\sum_{n\geq 0} b_n x^n$, where $a_n \in \mathbb{R}$ and $b_n > 0$ for all $n$. Further suppose that both series converge on $|x|<r$. If the sequence $\{a_n/b_n\}_{n\geq 0}$ is increasing (or decreasing), then the function $x \mapsto f(x)/g(x)$ is also increasing (or decreasing) on $(0,r)$.
\end{lemma}
We note that the above lemma still holds when both $f$ and $g$ are even, or both are odd functions.

\begin{theorem}\label{thm-4} Let $b \geq 0$ and $d, c >0$. Then follwoing assertions for ECHF are true.
\begin{enumerate}
\item[(i)] For $c \geq d$,  the function  $x \mapsto  \Phi_{\sigma}(b ; c; x)/\Phi_{\sigma}(b ; d; x)$ is increasing on $(0, \infty)$.
\item[(ii)] For $c \geq d$, we have  $d \Phi_{\sigma}(b+1 ; c+1; x) \Phi_{\sigma}(b ; d; x) \geq  c \Phi_{\sigma}(b ; c; x) \Phi_{\sigma}(b+1 ; d+1; x)$.
\item[(iii)]   The function  $x \mapsto  \Phi_{\sigma}(b ; c; x)$ is  log-convex  on $\mathbb{R}$.
\item [(iv)] The function $\sigma \mapsto  \Phi_{\sigma}(b ; c; x)$  is log-convex  on $(0, \infty)$ for fixed $x>0$.
\item[ (v)] Let $\delta >0$. Then  the function
 \[b \mapsto  \frac{B(b, c) \Phi_{\sigma}(b +\delta; c; x)}{B(b+\delta, c) \Phi_{\sigma}(b ; c; x)}\]
 is decreasing  on  $(0, \infty)$   for fixed $c, x>0$.
\end{enumerate}
\end{theorem}
\begin{proof}
From the definition of ECHF, it follows that
\begin{align}
\frac{\Phi_{\sigma}(b ; c; x)}{\Phi_{\sigma}(b ; d; x)} = \frac{ \sum_{n=0}^\infty \alpha_n(c) x^n } { \sum_{n=0}^\infty \alpha_n(d) x^n}, \quad \text{where} \quad  \alpha_n(t):=\frac{B_\sigma(b+n, t-b)}{B(b, t-b) n!}.
\end{align}
If we denote $f_n= \alpha_n(c)/ \alpha_n(d)$, then
\begin{align*}
 f_n- f_{n+1}
& = \frac{\alpha_n(c)}{ \alpha_n(d)}-\frac{\alpha_{n+1}(c)}{ \alpha_{n+1}(d)}\\
&= \frac{B(b, d-b)}{B(b, c-b)} \left( \frac{B_\sigma(b+n, c-b)}{B_\sigma(b+n, d-b)}-\frac{B_\sigma(b+n+1, c-b)}{B_\sigma(b+n+1, d-b)}\right).
\end{align*}
Now set $x:= b+n;$  $y:=d-b; $  $ x_1:=b+n+1$ and  $y_1:=c-b$ in  \eqref{eqn:exb-1}.  Since $(x-x_1)(y-y_1)=c-d \geq 0$, it follows from Theorem \ref{thm-1}  that
\[ \frac{B_\sigma(b+n, c-b)}{B_\sigma(b+n, d-b)} \leq \frac{B_\sigma(b+n+1, c-b)}{B_\sigma(b+n+1, d-b)},\]
which is equivalent to say that the sequence $\{f_n\}$ is increasing and appealing to Lemma \ref{lemma:1} ,  we can conclude that
$x \mapsto  \Phi_{\sigma}(b ; c; x)/\Phi_{\sigma}(b ; d; x)$ is increasing on $(0, \infty)$.

To prove $(ii)$, we need to recall the following identity from  \cite[Page 594] {Chaudhry-2} :
\begin{align}\label{eqn:diff-ext-echf}
\frac{d^n}{d x^n}  \Phi_{\sigma}(b ; c; x)= \frac{(b)_n}{(c)_n} \Phi_{\sigma}(b+n ; c+n; x).
\end{align}
Now the increasing properties of  $x \mapsto  \Phi_{\sigma}(b ; c; x)/\Phi_{\sigma}(b ; d; x)$ is equiavlent to
\begin{align}
\frac{d}{dx} \left(\frac{\Phi_{\sigma}(b ; c; x)}{\Phi_{\sigma}(b ; d; x)}\right)  \geq 0.
\end{align}
This together  with \eqref{eqn:diff-ext-echf}  implies
\begin{align*}
&\Phi'_{\sigma}(b ; c; x)  \Phi_{\sigma}(b ; d; x) -\Phi_{\sigma}(b ; c; x)  \Phi'_{\sigma}(b ; d; x)\\
& = \frac{b}{c} \Phi_{\sigma}(b+1 ; c+1; x) \Phi_{\sigma}(b ; d; x) - \frac{b}{d}\Phi_{\sigma}(b ; c; x)  \Phi_{\sigma}(b+1 ; d+1; x) \geq 0.
\end{align*}
A simple computation prove the assertion.

 The log-convexity of  $x \mapsto  \Phi_{\sigma}(b ; c; x)$ can be proved by using the integral representation of  ECHF as given in \eqref{eqn:ECHF-integral}   and by appealing to
the  H$\ddot{\text{o}}$lder-Rogers inequality for integrals as follows:
\begin{align*}
\Phi_\sigma(b ; c; \alpha x+ (1-\alpha) y)
 &=\frac{1}{B(b, c-b)}  \int_0^1 t^{b-1} (1-t)^{c-b-1} exp \left( \alpha x t+ (1-\alpha) y t - \tfrac{\sigma}{t(1-t)}\right)  dt\\
&=\frac{1}{B(b, c-b)}  \int_0^1\bigg[\left( t^{b-1} (1-t)^{c-b-1} exp \left(x t - \tfrac{\sigma}{t(1-t)}\right)\right)^\alpha \bigg.\\
& \hspace{1.5in} \bigg.\times  \left( t^{b-1} (1-t)^{c-b-1} exp \left( y t- \tfrac{\sigma}{t(1-t)}\right) \right)^{1-\alpha} \bigg] dt\\
&\leq \bigg[\frac{1}{B(b, c-b)}  \int_0^1 t^{b-1} (1-t)^{c-b-1} exp \left(x t - \tfrac{\sigma}{t(1-t)}\right) dt  \bigg]^\alpha \\
& \hspace{.4in} \times \bigg[\frac{1}{B(b, c-b)}  \int_0^1  t^{b-1} (1-t)^{c-b-1} exp \left( y t- \tfrac{\sigma}{t(1-t)}\right) dt \bigg]^{1-\alpha} \\
&= \left(\Phi_\sigma(b ; c;  x)\right)^\alpha  \left(\Phi_\sigma(b ; c;  y)\right)^{1-\alpha} ,
\end{align*}
here $ x, y \geq 0$ and $\alpha \in [0,1]$. This prove that $x \mapsto  \Phi_{\sigma}(b ; c; x)$ is log-convex for $x\geq 0$. For the case $x <0$, the assertion follows immediately from the identity (see.\cite[Page 596]{Chaudhry-2}):
\[ \Phi_{\sigma}(b ; c; -x)= e^{-x} \Phi_{\sigma}(c-b; c; x). \]

It is known that the infinite sum of log-convex functions is also log-convex. Thus, the log-convexity of  $\sigma \mapsto  \Phi_{\sigma}(b ; c; x)$  is equivalent to show that $\sigma \mapsto B_\sigma(b+n, c-b)$ is log-convex on $(0, \infty)$ and for all non-negative integer $n$.  From Theorem \ref{thm-2} , it is clear that $\sigma \mapsto B_\sigma(b+n, c-b)$ is log-convex for $c>b>0$ and hence  (iv) is true.

 Let $b' \geq b$. Set  $p(t):=  t^{b'-1} (1-t)^{c-b'-1} exp \left(x t- \tfrac{\sigma}{t(1-t)}\right)$,
\[
  f(t):= \left(\frac{t}{1-t}\right)^{b-b'} \quad  \text{and} \quad g (t):= \left(\frac{t}{1-t}\right)^{\delta}.
\]
Then using the integral  representation \eqref{eqn:ECHF-integral} of ECHF, we have
\begin{align}\label{eqn:22}
\frac{B(b, c) \Phi_\sigma(b+\delta ; c; x) }{B(b+\delta, c)  \Phi_\sigma(b ; c; x)} - \frac{B(b', c) \Phi_\sigma(b'+\delta ; c; x) }{B(b'+\delta, c)  \Phi_\sigma(b' ; c; x)}
=\frac{\int_0^1 f(t) g(t) p(t)  dt}{\int_0^1 f(t) p(t) dt } - \frac{\int_0^1 g(t) p(t) dt}{ \int_0^1p(t) dt }.
\end{align}
It is easy to determine that for $b' \geq b$, the function $f$  is decreasing while as $ \delta \geq 0$, the function $g$ is increasing. Since $p$ is non-negative for $t \in [0,1]$, by the reveres
Chebyshev's integral inequality  \eqref{eqn:chebyshev-1}, it follows that
\begin{align}
\int_0^1 p(t) f(t) dt \int_0^1 p(t) g(t) dt \leq \int_0^1 p(t)  dt \int_0^1 p(t)f(t) g(t) dt.
\end{align}
This together with  \eqref{eqn:22} implies
\begin{align*}
\frac{B(b, c) \Phi_\sigma(b+\delta ; c; x) }{B(b+\delta, c)  \Phi_\sigma(b ; c; x)} - \frac{B(b', c) \Phi_\sigma(b'+\delta ; c; x) }{B(b'+\delta, c)  \Phi_\sigma(b' ; c; x)}
\geq 0,
\end{align*}
which is equivalent to say that the function
 \[b \mapsto  \frac{B(b, c) \Phi_{\sigma}(b +\delta; c; x)}{B(b+\delta, c) \Phi_{\sigma}(b ; c; x)}\]
is decreasing on $(0, \infty)$.
\end{proof}

\begin{remark}
In particular,   the decreasing properties of
 \[b \mapsto  \frac{B(b, c) \Phi_{\sigma}(b +\delta; c; x)}{B(b+\delta, c) \Phi_{\sigma}(b ; c; x)}\]
is equivalent to the inequality
\begin{align}\label{eqn-23}
 \Phi_\sigma^2(b+\delta ; c; x) \geq \frac{ B^2(b+\delta, c)}{ B(b+2 \delta, c) B(b, c)}  \Phi_\sigma(b+2\delta ; c; x) \Phi_\sigma(b; c; x).
\end{align}
Now define
\[ f(\delta) :=  \frac{ B^2(b+\delta, c)}{ B(b+2 \delta, c) B(b, c)}= \frac{ (\Gamma(b + \delta))^2 \Gamma(b + 2 \delta+c) \Gamma(b+c)}{(\Gamma( b+c + \delta))^2 \Gamma(b+ 2 \delta) \Gamma(b)}.\]
A logarithmic differentiation of $f$ yields
\begin{align*}
\frac{f'(\delta)}{f(\delta)} = 2 \psi (b + \delta) + 2 \psi (b + 2 \delta+ c)- 2 \psi (b+c + \delta) - 2 \psi (b + 2 \delta),
\end{align*}
where  $y \mapsto\psi(y)=\Gamma'(y)/\Gamma(y)$ is  the digamma function which is  increasing on $(0,\infty)$ and it
 has the series form
\[\psi(y)=-\gamma+ \sum_{k\geq 0}\left(\frac{1}{k}-\frac{1}{y+k}\right).\]

This implies
\begin{align*}
\frac{f'(\delta)}{f(\delta)}
 &= 2 \sum_{k=0}^\infty \left( \frac{1}{b+c+\delta+k}-\frac{1}{b+c+2\delta+k}\right)+ 2 \sum_{k=0}^\infty \left( \frac{1}{b+\delta+k}-\frac{1}{b+2\delta+k}\right)\\
&= 2 \delta \sum_{k=0}^\infty \left( \frac{1}{(b+c+\delta+k)(b+c+2\delta+k) }-  \frac{1}{(b+\delta+k)(b+2\delta+k)}\right)\\
&= - 2 \delta \sum_{k=0}^\infty \frac{c(2b+3 \delta+ 2 k+c)}{(b+c+\delta+k)(b+c+2\delta+k) (b+\delta+k)(b+2\delta+k)} \leq 0.
\end{align*}
Thus, $f$ is decreasing functions of $\delta$ on $[0, \infty)$ and $f(\delta) \leq f(0)=1$.

Interestingly, for  $\sigma=0$ the inequality in \eqref{eqn-23} reduce  to the  Tur\`an type inequality of classical confluent hypergeometric functions as
\begin{align}\label{eqn-24}
{}_1F_{1}^2(b+\delta ; c; x) \geq \frac{ B^2(b+\delta, c)}{ B(b+2 \delta, c) B(b, c)}  {}_1F_{1}(b+2\delta ; c; x) {}_1F_{1}(b; c; x).
\end{align}
By the fact that
 \[\frac{ B^2(b+\delta, c)}{ B(b+2 \delta, c) B(b, c)} \leq 1\]
it can be conclude that the inequality \eqref{eqn-24} is an  improvement of the inequality given in \cite[Theorem 4(b)]{KS} for fixed $c, x>0$. However, our result does not expound the other cases  in \cite[Theorem 4(b)]{KS}.

Now following the  remark given in \cite[Page-390]{KS} for integer $\delta$ and $b=\delta+a$  in  \eqref{eqn-24},  will also improve the  inequality \cite[Theorem 1, Corollary 2]{BGR} for classical confluent  hypergeometric functions.
\end{remark}

Our next result is on the extended Gaussian hypergeometric functions(EGHF).

\begin{theorem} Let $b \geq 0$ and $d, c >0$. Then follwoing assertions for EGHF are true.
\begin{enumerate}
\item[(i)] For $c \geq d$,  the function  $x \mapsto  F_{\sigma}(a, b ; c; x)/F_{\sigma}(a, b ; d; x)$ is increasing on $(0, 1)$.
\item[(ii)] For $c \geq d$, we have  \[d F_{\sigma}(a+1, b+1 ; c+1; x) F_{\sigma}(a, b ; d; x) \geq  c F_{\sigma}(a+1, b+1 ; d+1; x) F_{\sigma}(a, b ; c; x).\]
\item[(iii)]   The function $\sigma \mapsto  F_{\sigma}(a, b ; c; x)$  is log-convex  on $(0, \infty)$ for fixed $b>0, c>0$  and $x \in (0, 1)$.
\item[(iv)]  The function $a \mapsto  F_{\sigma}(a, b ; c; x)$  is log-convex  on $(0, \infty)$ and  for fixed $x \in (0, 1)$.
%\item [(iv)]  The function  $x \mapsto F_{\sigma}(a, b ; c; x)$ is  log-convex  on $(0, 1)$.
\end{enumerate}
\end{theorem}

\begin{proof}
The case   $(i)$-$(iii)$ can  be proved by following the proof of Theorem \ref{thm-4}  by considering the series form
\eqref{eqn:EGHF-1} and  an integral representation
\eqref{eqn:EGHF-integral}  of EGHF , We omit the details.

From a result of  Karp and Sitnik \cite{KS},  we know that if we let
\[f(a,x)=\sum_{n\geq0}f_n \frac{(a)_n}{n!} x^n,\]
where $f_n$ (and is independent of $a$) and we suppose that $a'>a>0$, $\delta >0,$ then the function
\[f(a+\delta,x)f(b,x)-f(b+\delta,x)f(a,x)=\sum_{m\geq2}\phi_m x^m\]
has negative  power series coefficient $\phi_m < 0$ so that $a \mapsto f(a,x)$ is strictly log-convex for $x>0$ if the sequence $\left\{f_n/f_{n-1}\right\}$ is increasing. In what follows we shall use this result for the function $F_\sigma(a, b; c; x).$ For this let
\[f_n=\frac{ B_\sigma(b+n, c-b)}{B(b, c-b)}  .\]
Thus, to prove $(iv)$, it is enough to show that the sequence $d_n = f_n/ f_{n-1}$ is decreasing. Clearly
\[d_n- d_{n-1} = \frac{B_\sigma(b+n, c-b)}{B_\sigma(b+n-1, c-b)}- \frac{B_\sigma(b+n-1, c-b)}{B_\sigma(b+n-2, c-b)}.\]
Now if we replace  $x_1, y_1, x_2, y_2$ in \eqref{eqn1} by  $x_1=b+n$, $x_2=b+n-2$, and $y_1=y_2=c-b$, then  it  follows that $ d_n \geq d_{n-1}$. Hence the conclusion.
\end{proof}
\section{Conclusion}
This article proves several properties of the extended beta functions resembles with the classical beta functions. Few of those properties are key to establish inequalities for ECHF and EGHF. Using classical integral inequalities, Tur\'an type and  reverse
Tur\'an type inequality for ECHF and EGHF are also given.

\section*{List of abbreviations}
\text{ECHF: Extended confluent hypergeometric functions,}\\
\text{EGHF: Extended Gaussian hypergeometric functions.}

\paragraph{\bf Competing interests.}
The authors declare that they have no competing interests

\paragraph{\bf Authors details.}
Assistant Professor, Department of Mathematics and Statistics, Collage of Science,
King Faisal University, Al-Hasa 31982, Hofuf, Saudi Arabia.
\email{smondal@kfu.edu.sa}

\paragraph{\bf Acknowledgements.} The work was supported by the Deanship of Scientific Research, King Faisal University, Saudi Arabia through the project no- $150244$.

 \end{document}